\documentclass[12pt]{amsart}

\usepackage{fancyhdr}
\usepackage[active]{srcltx}
\usepackage{amsmath,amssymb}
\usepackage{longtable}
\usepackage[arrow,matrix,curve]{xy}
\usepackage{amstext}
\usepackage{amssymb}
\usepackage{amsfonts}
\usepackage{amssymb}
\usepackage{amsmath}
\usepackage{graphicx}
\usepackage{geometry}

 \newtheorem{thm}{Theorem}

 \newtheorem{prop}[thm]{Proposition}
 
 \newtheorem{rem}{Remark}
 \newtheorem{ex}{Example}
 \numberwithin{equation}{section}

\newcommand{\Sc}[2]{\langle #1,#2\rangle}
\newcommand{\Hom}{\mathrm{Hom}}

\newcommand{\Aut}{\mathrm{Aut}}
\newcommand{\Ext}{\mathrm{Ext}}

\newcommand{\RQ}[1]{\textrm{Rep}_Q(#1)}
\newcommand{\Rep}[1]{\RQ{\mathbb F_{#1}}}

\newcommand{\Hu}[1]{\mathbf{H}_{#1}}
\newcommand{\Ht}[1]{\mathbf{H}^{tw}_{#1}}

\newcommand\Zn{\mathbb{Z}}
\newcommand\Nn{\mathbb{N}}
\newcommand\Cn{\mathbb{C}}

\begin{document}

\title{Some remarks on Hall algebra of bound quiver}

\author{Kostiantyn Iusenko}
\address[Kostiantyn Iusenko]{Departamento de Matem\'{a}tica Univ. de S\~ao Paulo, Caixa Postal 66281, S\~ao Paulo, SP 05315-970 -- Brazil}
\email{iusenko@ime.usp.br}

\author{Evan Wilson}
\address[Evan Wilson]{Departamento de Matem\'{a}tica Univ. de S\~ao Paulo, Caixa Postal 66281, S\~ao Paulo, SP 05315-970 -- Brazil}
\email{wilsoneaster@gmail.com}

\begin{abstract} 
In this paper we describe the twisted Hall algebra of bound quiver with small homological dimension. The description is given in the terms of the quadratic form associated with the corresponding bound quiver. 
\end{abstract}

\maketitle

\section*{Introduction}

Let $Q$ be a quiver, $\mathfrak g$ be a symmetric Kac-Moody algebra associated with $Q$, and $\Rep{q}$ be the category of finite-dimensional representations of $Q$ over $\mathbb F_q$, the field with $q=p^n$ elements.  In his remarkable papers \cite{rin,rin2} Ringel proved that if $Q$ is a Dynkin quiver then there exists an isomorphism between the (twisted) Hall algebra associated with $\Rep{q}$ and the positive part $U_t^+(\mathfrak{g})$ of quantized universal enveloping algebra with $t^2=q$. 

In this paper we consider the case of a quiver $Q$ bound by an admissible ideal $I$. To each such quiver we associate an associative algebra $U_t^+(Q)$ given by relations and generators. In the case when $\mathbb F_q Q/I$ is a representation directed algebra of global dimension at most 2 we show that there exists a homomorphism $\rho$ between $U_t^+(Q)$ and the corresponding twisted Hall algebra $\Ht{\Rep{q}}$, in which $\Rep{q}$ is a category of finite-dimensional bound representations of $Q$. In the case when $q\neq 2$ we show that $\rho$ is an isomorphism (see Section 2 for details).

In final section we state some concrete examples of bound quivers and corresponding twisted Hall algebras and compare these examples with the ones obtained in \cite{ChDe}.

\section*{Acknowledgements}This work was done during the visit of the authors to the University of S\~{a}o Paulo as a postdoctoral fellows. We are grateful to the University of S\~{a}o Paulo for hospitality, our supervisor Vyacheslav Futorny for stimulating discussion and to the Fapesp for financial support (grant numbers 2010/15781-0 and 2011/12079-5).

\section{Preliminaries} 
\subsection{Hall algebras}
Let $\mathcal A$ be an $\mathbb F_q$-linear finitary category.  We define the (Ringel) Hall algebra $\Hu{\mathcal A}$ to be the $\mathbb C$-vector space spanned by  elements $M$ in $\textrm{Iso}(\mathcal A)$ (the set of isomorphism classes of the objects in $\mathcal A$) with a product defined by:
\begin{equation*}
[M] \ast [N]=\sum_{R\in \Hu{\mathcal{A}}}  F_{M,N}^R [R]
\end{equation*}
where $F_{M,N}^R$ denotes the number of the sub-objects $X \subset R$ such that $X\simeq N$ and $R/X \simeq M$.
Let $K_0(\mathcal A)$ be Grothendieck group of the category $\mathcal A$. By $\langle -,- \rangle: K_0(\mathcal A) \times K_0(\mathcal A) \rightarrow \mathbb Z$ we denote the Euler characteristic of $\mathcal A$: 
$$
	\langle M,N\rangle = \sum_{k=0}^{\infty}(-1)^k \dim \Ext^k(M,N).
$$
The twisted version of the product is defined by 
\begin{equation*}\label{hallp}
[M]\cdot [N]=q^{\frac{1}{2}\Sc{M}{N}}\sum_{R\in \Hu{\mathcal{A}}} F_{M,N}^R [R]
\end{equation*}
Both products give $\Hu{\mathcal{A}}$ an associative algebra structure (see \cite{rin}). The twisted version of Hall algebra will be denoted by $\Ht{\mathcal A}$. 

We mention that both $\Hu{\mathcal{A}}$ and $\Ht{\mathcal{A}}$ possess a grading by the Grothendieck group $K_0(\mathcal A)$ (see \cite{rin}).  For $\alpha\in K_0(\mathcal A)$ we denote the degree $\alpha$ piece by $\Hu{\mathcal{A}}[\alpha]$ and by $\Ht{\mathcal{A}}[\alpha]$. 

For more information on Hall algebras the interested reader can see, for example, \cite{OS} and references therein.

\subsection{Lie Algebras Associated with Unit Forms}
Recall the class of Lie algebras associated with unit forms considered by Kosakowska \cite{kos}.  A unit form is a mapping $T:\Zn^{I}\to \Zn$ for some finite index set $I$ which is of the following form:
\begin{equation*}
T(\beta)=\sum_{i\in I} \beta_i^2+\sum_{(i,j)\in I\times I}a_{ij}\beta_i\beta_j
\end{equation*}
where $a_{ij}\in \Zn$.  An example is given by the form associated to a quiver defined in Section \ref{quiver}.  A root of $T$ is a vector $\beta \in \Zn^{I}$ satisfying $T(\beta)=1$.  A root $\beta$ is called positive if $\beta_i\geq 0,\ i \in I$.  We denote the set of roots of $T$ by $\Delta_T$ and the set of positive roots of $T$ by $\Delta^+_T$.  

Define the Lie algebra $L(T)$ to be the free Lie algebra with generators $\{e_i: \ i\in I\}$ modulo the ideal generated by the set of all elements of the form \begin{equation}  \label{multicomm}
[e_{i_1},[e_{i_2},[\cdots [e_{i_{k-1}},e_{i_k}]\cdots]]]
\end{equation}
such that $$\sum_{j=2}^k\alpha_{i_j}\in \Delta_T,\quad \mbox{but}\quad  \sum_{j=1}^k\alpha_{i_j}\notin \Delta_T,$$ where $\{\alpha_i;\ i\in I\}$ denotes the standard basis of $\Zn^I$.  The expression \eqref{multicomm}  appearing above is called a standard multicommutator.  This algebra receives an $\Nn^{I}$-grading where each generator $e_i$ has degree $\alpha_i$.

\begin{rem}
The Lie algebra $L(T)$ was shown \textup{(\cite[Proposition 4.4]{kos})} to be isomorphic to $G^+(T)$--the positive part of the Lie algebra studied by Barot, Kussin, and Lenzing in \cite{bkl}--in the case that $T$ is both weakly positive and positive semidefinite.
\end{rem}

\begin{rem}
Assuming that the form $T$ is positive definite the minimal set of defining relations in algebra $L(T)$ was constructed in \textup{\cite[Theorem 1.1]{kos}}.
\end{rem}

\subsection{Quantized Lie Algebras associated with Unit forms}
Let $T$ be a unit form and $\Delta^+_{T}$ its set of its positive roots. 
Let $\langle \_,\_ \rangle_T : \Zn^{I}\times \Zn^{I}\to \Zn$ be a bilinear form associated with $T$: 
$$
\langle \beta,\beta' \rangle_T:=\sum_{i\in I}\beta_i\beta'_i+\sum_{(i,j)\in I\times I}a_{ij}\beta_i\beta'_j.
$$
Denote by 
$$
\langle \beta,\beta' \rangle_T^{0}=\sum_{i\in I}\beta_i\beta'_i+\sum_{(i,j)\in I\times I}(a_{ij})_-\beta_i\beta'_j,
$$
where $(a_{ij})_-=\min\{a_{ij},0\}$. We consider the function $\nu:\Zn^{I}\times \Zn^{I}\to \Zn$ defined by 
\begin{equation} \label{nuDef}
	\nu(\beta,\beta')=\delta\Big (\sum_{(i,j)\in I\times I}(a_{ij})_-\beta_i\beta'_j\Big ) \langle\beta, \beta'\rangle_T^0,
\end{equation}
in which $\delta:\mathbb Z\rightarrow \{0,1\}$ is defined as $\delta(0)=1$ and $\delta(x)=0$ elsewhere.

\noindent The free associative algebra $\mathbb C\langle e_i; i\in I\rangle$ is $\Nn^{I}$-graded by giving $e_i$ degree $\alpha_i$.  Fixing $t\in \Cn$ we define the map 
\begin{equation} \label{adDef}
\text{ad}^t_{x}(y)=xy-t^{\langle \beta,\beta'\rangle_T-\langle \beta',\beta\rangle_T +2\nu(\beta',\beta)-2\nu(\beta,\beta')}yx
\end{equation}
for homogeneous elements $x$ and $y$ of degree $\beta$ and $\beta'$ respectively.
This map depends on the form $T$, but in what follows we use notation $\text{ad}^t_{x}(y)$ assuming that it is always clear which form we have. Note that if $\langle \_,\_\rangle_T=\langle \_,\_\rangle_T^0$ then \eqref{adDef} gives $$\text{ad}_x^t(y)=xy-t^{\pm \langle \beta,\beta'\rangle_T^0\pm\langle \beta',\beta\rangle_T^0}yx.$$  
This is slightly different from the quantized adjoint action: $xy-t^{\langle\beta,\beta'\rangle_T^0+\langle\beta',\beta\rangle_T^0}yx$. Our version is useful for studying the Hall algebra associated to a quiver, as we will see in the next section.

Consider the following family of associative algebras 
$$
	U_t^+(T)=\mathbb C\langle e_i; i\in I\rangle / \mathcal R^t_T,
$$
where $\mathcal R^t_T$ is the two-sided ideal generated by the set of all elements
\begin{equation}\label{adRel}
\text{ad}^t_{e_{i_1}}(\text{ad}^t_{e_{i_2}}(\cdots \text{ad}^t_{e_{i_{k-1}}}(e_{i_k})\cdots)),
\end{equation}
such that $\sum_{j=2}^k\alpha_{i_j}\in \Delta^+_T$ but $\sum_{j=1}^k\alpha_{i_j}\notin \Delta^+_T$.
\begin{rem} \label{Rem:spec}
The specialization of $U_t^+(T)$ at $t=1$ is isomorphic to the universal enveloping algebra of $L(T)$.
\end{rem}

\begin{rem}
Using Remark 1 one can show that in the case 
where $T$ is both weakly positive and positive semidefinite, $U_t^+(T)$ gives a quantization of a positive part of $G^+(T)$ studied in \cite{bkl}. 
\end{rem}

Here and further, for an $\mathbb N^I$ graded algebra $A$, we denote the $\alpha$-homogeneous piece by $A[\alpha]$ for $\alpha\in \mathbb N^I$.
\begin{prop} \label{equalDim}
	Let $\alpha \in \Nn^I$ and $t$ be any complex number then 
	$$\dim_{\mathbb C} U_t^+(T)[\alpha]=\dim_{\mathbb C}U_1^+(T)[\alpha]=\dim_{\mathbb C} U(L(T))[\alpha].$$
\end{prop}

\begin{proof}
One can use standard arguments coming from \cite{rein} or \cite{rin2} for instance. 

Denote by $A$ the Laurent polynomial ring $\mathbb C[u,u^{-1}]$. Then $A/(u-1) \cong \mathbb C$ under the isomorphism of evaluation at $u=1$. Let $$U_A^+(T):=A \langle e_i; i\in I\rangle / \mathcal R^u_T.$$ We have:
$$
	U_1^+(T)= A/(u-1) \otimes_A U_A^+(T) \cong \mathbb C \otimes_A U_A^+(T).
$$

But $U_1^+(T)=U(L(T))$ and hence has a PBW basis with finite dimensional $\mathbb{N}^{I}$ homogeneous subspaces, which follows directly from Remark \ref{Rem:spec} since $U_1^+(T)$ is exactly the specialization of $U_t^+(T)$ at $t=1$. Also, $U_A^+(T)[\alpha]$ is finitely generated as an $A$-module, $\alpha \in \mathbb N^I$. It remains to be shown that $U_A^+(T)[\alpha]$ is free as an $A$-module, since in that case we have:
$$
	U_1^+(T)[\alpha] \cong \mathbb C\otimes_A U_A^+(T)[\alpha] \cong 
	(\mathbb C\otimes_A A^s) \cong \mathbb C^s,\quad  \mbox{for some}\ \ s\in \mathbb N_0.
$$
Let $U^*=\mathbb{C}(u)\otimes_A U_A^+(T)$. Then $U_A^+(T)[\alpha]$ is a submodule of $U^*[\alpha]$ considered as $A$-modules, hence is torsion free. Therefore, since $A$ is a principal ideal domain we see that $U_A^+(T)[\alpha]$ is free, which completes the proof.
\end{proof}

\section{Category $\Rep{q}$ and its twisted Hall algebra}

\subsection{Category of bound representations of a quiver} Let $Q$ be a quiver given by a set of vertices $Q_0$ and a set of arrows $Q_1$ denoted by $\rho:i\rightarrow j$ for $i,j\in Q_0$. We only consider finite quivers without oriented cycles, loops, or multiple arrows. 

A finite-dimensional $\mathbb F_q$-representation of $Q$ is given by a tuple \[V=((V_i)_{i\in Q_0},(V_{\rho})_{\rho \in Q_1}:V_i\rightarrow V_j)\] of finite-dimensional $\mathbb F_q$-vector spaces and $\mathbb F_q$-linear maps between
them. A  morphism of representations $f:V\rightarrow W$ is a tuple $f=(f_i:V_i\rightarrow W_i)_{i\in Q_0}$ of $\mathbb F_q$-linear maps such that $W_{\rho}f_i=f_{j}W_{\rho}$ for all $\rho:i\rightarrow j$. 

Let $\mathbb F_q Q$ be the path algebra of $Q$ and $I$ be an admissible ideal in $\mathbb F_q Q$ (see \cite[Section II.1]{ass} for precise definitions). Supposing that $I$ is generated by a minimal set of relations, we denote by $r(i,j)$ the number of relations with starting vertex $i$ and terminating vertex $j$.  We say that the representation $V$ of $Q$ is a bound by  $I$ if $V_r=0$ for any $r\in I$ (that is, all the relations in $I$ are satisfied). To simplify the notation we denote by $\Rep{q}$ the abelian category of representations of $Q$ over $\mathbb F_q$ bound by an ideal $I$ assuming that it is always clear which ideal we have. One can show that $\Rep{q}$ is equivalent to the category of left modules over $\mathbb F_qQ/I$ analogously to the case of unbound quivers (see \cite[Section III.1]{ass}). Note that, in general, $\Rep{q}$ is not hereditary, i.e. we can have $\Ext^{k}(V,W)\neq 0$ for some $k>1$ and some representations $V$ and $W$. 

The Grothendieck group $K_0(\Rep{q})$ can be naturally identified with $\mathbb{Z}^{Q_0}$. The dimension $\underline{\dim} V\in K_0(Q)$ of $V$ is defined by $\underline{\dim}V=(\dim V_i)_{i\in Q_0}$.

\subsection{Representation directed quivers}  \label{quiver}

In the following we assume that the algebra $\mathbb F_q Q/I$ is of global dimension at most two. In this case for two representations $V$ and $W$ with $\underline{\dim} V=\beta$ and $\underline{\dim} W=\beta'$ the Euler form of $\Rep{q}$ has a precise form given in terms of dimension vectors (see \cite[Proposition 2.2]{bon})
\begin{equation*} 
\begin{split}
\Sc{V}{W}=&\dim\Hom(V, W)-\dim\Ext^1(V,W)+\dim\Ext^2(V,W)\\
=&\sum_{i\in Q_0}\beta_i\beta_i'-\sum_{\rho:i\rightarrow j\in Q_1} \beta_i\beta_j'+\sum_{(i,j)\in Q_0\times Q_0}r(i,j)\beta_i \beta'_j,
\end{split}
\end{equation*}
hence the Euler form gives rise to the following unit form 
$$
	T_Q(\beta):=\Sc{\beta}{\beta}=\sum_{i\in Q_0}\beta_i^2-\sum_{\rho:i\rightarrow j\in Q_1} \beta_i\beta_j+\sum_{(i,j)\in Q_0\times Q_0}r(i,j)\beta_i \beta_j.
$$

In what follows we only consider the case when the algebra $\mathbb F_q Q/I$ is representation directed, which, in particular, means that the bound quiver $(Q,I)$ has just a finite number (up to equivalence) of indecomposable representations. In this case we can enumerate the indecomposable representation $V^{(1)},\dots,V^{(n)}$ in such a way that $\Ext^1(V^{(k)}, V^{(l)})=0$, if $k \leq l$ and $\Hom(V^{(k)},V^{(l)})=0$, if $k<l$.

\begin{rem} \label{GabAn} 
Due to \cite{bon} if $\mathbb F_q Q/I$ is representation directed then the form ${T_Q}$ is weakly positive and the set of its  positive roots  $\Delta_{T_Q}^+$ is finite. Moreover by \cite[Theorem 1.3]{dr} we have that if $q\neq 2$ then the assignment $V \mapsto \underline{\dim} V$ gives rise to a bijection between the set of indecomposable representations of $Q$ and $\Delta_{T_Q}^+$ (see also \cite[Theorem 3.3]{bon} for the original statement for algebraically closed fields).
\end{rem}

We mention also that Ringel showed that for any representation directed algebra there exist Hall polynomials which count Hall numbers $F_{M,N}^{R}$ (see \cite[Theorem 1]{rin2} for details). Let $S_i$ denote the simple representation of $Q$ at vertex $i\in Q_0$, i.e. $(S_i)_i=\mathbb F_q$, $(S_i)_j=0$ for $j\neq i$ and $(S_i)_\rho=0$ for all $\rho \in Q_1$. We will need the following

\begin{prop} \label{HallCon}
Let $M$ be an indecomposable bound representation of $Q$ over $\mathbb F_q$. The following identity holds
$$F^{M\oplus S_i}_{S_i,M}=q^{\nu(m,\alpha_i)-\nu(\alpha_i,m)}F^{M\oplus S_i}_{M,S_i},$$
in which $i \in Q_0$, $m=\underline{\dim}M$ and $\nu$ is a function given by \eqref{nuDef}.
\end{prop}

\begin{proof}
First we mention that for any $V, W, R \in \Rep{q}$ we have
$$
	F^R_{V,W}=\frac{|\Ext(V,W)_R|}{|\Hom(V,W)|}\cdot \frac{|\Aut(R)|}{|\Aut(V)| \cdot |\Aut(W)|},
$$
in which $\Ext(V,W)_R$ is the subset of $\Ext(V,W)$ parameterising extensions with middle term isomorphic to $R$.

\noindent It is straightforward to see that 
$$|\Hom(M,S_i)|=q^{\nu(m,\alpha_i)}, \quad |\Hom(S_i,M)|=q^{\nu(\alpha_i,m)}.$$
Also we have that $$|\Ext(M,S_i)_{M\oplus S_i}|=|\Ext(S_i,M)_{M\oplus S_i}|=1.$$ 
Therefore
$$
	\frac{F^{M\oplus S_i}_{S_i,M}}{F^{M\oplus S_i}_{M,S_i}}=\frac{|\Hom(M,S_i)|}{|\Hom(S_i,M)|}=\frac{q^{\nu(m,\alpha_i)}}{q^{\nu(\alpha_i,m)}}=q^{\nu(m,\alpha_i)-\nu(\alpha_i,m)}.
$$
\end{proof}

\subsection{Twisted Hall algebra of $\Rep{q}$}

Let $Q$ be a bound quiver of global dimension at most two. Fixing some $t \in \mathbb C$ consider the quantized Lie algebra $U_t^+(T_Q)$ where $T_Q$ is the unit form associated with $Q$.

\begin{rem}
Suppose that $i\to j \in Q_1$.  Then we have $T_Q(\alpha_i+\alpha_j)=1$, hence $\alpha_i+\alpha_j\in \Delta^+_{T_Q},$ but $2\alpha_i+\alpha_j\notin \Delta^+_{T_Q}$.  We also have $\langle \alpha_i,\alpha_j\rangle -\langle \alpha_j,\alpha_i\rangle=-1$ and $\langle \alpha_i,\alpha_i+\alpha_j\rangle -\langle \alpha_i+\alpha_j,\alpha_i,\rangle+2=1$.  Then we have the following relation in $U_t^+(T_Q)$
\begin{align*}
\textup{ad}^t_{e_i}(\textup{ad}^t_{e_i}(e_j))&=\textup{ad}^t_{e_i}(e_ie_j-t^{-1}e_je_i)\\
	&=e_i^2e_j-(t+t^{-1})e_ie_je_i+e_je_i^2=0
\end{align*}
which is the usual quantum Serre relation for simple roots connected by a single arrow.  A similar statement holds for $\textup{ad}^t_{e_j}(\textup{ad}^t_{e_j}(e_i))$.
\end{rem}

One can also see that if $Q$ is unbound then $U_t^+(T_Q)$ is a quantized universal enveloping algebra of the positive part of the Lie algebra $\mathfrak g$ associated with a quiver $Q$. As proved by Ringel (see \cite{rin}) in this case the assignment $e_i \mapsto [S_i]$ (where $S_i$ is a simple representation at vertex $i\in Q_0$) gives rise to a homomorphism $\rho$ between $U_t^+(T_Q)$ (which equals  $U_t^+(\mathfrak g)$ in this case) and $\Ht{\Rep{q}}$ with $t=+\sqrt{q}$. Moreover if $Q$ has finite type then $\rho$ is an isomorphism.  

We now prove a similar theorem for bound case.

\begin{thm}
Let $Q$ be a quiver and $I$ an admissible ideal in $\mathbb F_q Q$ such that $\mathbb F_qQ/I$ is of global dimension at most 2 and representation directed.
Then the assignment $e_i \mapsto [S_i]$, $i\in Q_0$ gives rise to a homomorphism of associative algebras 
{\normalfont $$\rho:U_t^+(T_Q)\to \Ht{\Rep{q}},$$}
with $t=+\sqrt{q}$. Moreover if $q\neq 2$ then $\rho$ is an isomorphism.
\end{thm}
\begin{proof}
We show $\rho$ is an homomorphism. We need to show that the relations \eqref{adRel} are satisfied. Let $M$ be an indecomposable representation with dimension $\underline{\dim}M=m$. Suppose that $m+\alpha_i$ is not a root.  Therefore 
\begin{align*}
	[M]\cdot[S_i]&=q^{\frac{1}{2}\Sc{m}{\alpha_i}} F_{M,S_i}^{M\oplus S_i} [M \oplus S_i],\\
	[S_i]\cdot[M]&=q^{\frac{1}{2}\Sc{\alpha_i}{m}} F_{S_i,M}^{M\oplus S_i} [M \oplus S_i],
\end{align*}
By Proposition \ref{HallCon} we have that 
\begin{align*}
	[S_i]\cdot[M]&=q^{\frac{1}{2}\Sc{\alpha_i}{m}} q^{\nu(\alpha_i,m)-\nu(m,\alpha_i)} F_{M,S_i}^{M\oplus S_i} [M \oplus S_i]\\
	 &=q^{\frac{1}{2}\Sc{\alpha_i}{m}+\nu(m,\alpha_i)-\nu(\alpha_i,m)}F_{M,S_i}^{M\oplus S_i} [M \oplus S_i].	 
\end{align*}
Thus we have that $\text{ad}^t_{[S_i]}([M])=0.$

Suppose that $m+\alpha_i$ is a root. Then (since $\mathbb F_qQ/I$ is representation directed) there exists exactly one indecomposable representation with dimension $m+\alpha_i$. Denote this representation by $M_1$. Also, we have that either $\Ext(S_i,M)=0$ or $\Ext(M,S_i)=0$. Assume for definiteness that $\Ext(S_i,M)=0$. Then 
\begin{align*}
	[M]\cdot[S_i]&=q^{\frac{1}{2} \Sc{m}{\alpha_i}}\left(F_{M,S_i}^{M\oplus S_i} [M \oplus S_i]+F_{M,S_i}^{M_1}[M_1]\right),\\
	[S_i]\cdot[M]&=q^{\frac{1}{2} \Sc{\alpha_i}{m}} F_{S_i,M}^{M\oplus S_i} [ M \oplus  S_i]\\
	&=q^{\frac{1}{2} \Sc{\alpha_i}{m}+\nu(m,\alpha_i)-\nu(\alpha_i,m)} F_{M,S_i}^{M\oplus S_i} [ M \oplus  S_i].
\end{align*}
Now we consider $$\text{ad}^t_{[S_i]}([M])=[S_i]\cdot [M]-t^{\Sc{\alpha_i}{m}-\Sc{m}{\alpha_i}+2\nu(m,\alpha_i)-2\nu(\alpha_i,m)}[M]\cdot[S_i].$$  The coefficient of $[M\oplus S_i]$ is 0 and hence we are left with $\text{ad}^t_{[S_i]}([M])=c[M_1],$ where $c\neq 0 \in \mathbb C$.

Now suppose that $\sum_{j=2}^k\alpha_{i_j}\in \Delta^+_{T_Q}$ then inductively using 
$\text{ad}^t_{[S_{i_{j-1}}]}([M_{j}])=c_{i_{j-1}}[M_{j-1}]$ where $M_{j},j\in \{2,3,\dots k\}$ is the unique indecomposable of degree $\sum_{\ell=j}^k\alpha_{i_\ell}$ and $c_{i_j}\in \mathbb C$ we conclude that 
\begin{equation*}
\text{ad}^t_{[S_{i_2}]}(\text{ad}^t_{[S_{i_3}]}(\cdots \text{ad}^t_{[S_{i_{k-1}}]}([S_{i_k}])\cdots))=c[M]
\end{equation*}
where $M=M_2$ and $c\in \Cn$ (possibly 0). But then $\text{ad}^t_{[S_{i_1}]}([M])=0$ as $\sum_{j=1}^k\alpha_{i_j}\notin \Delta^+_{T_Q}$. Which proves that $\rho$ is a homomorphism.

Assume now that $q\neq 2$. Since $\mathbb F_q Q/I$ is representation directed by Remark \ref{GabAn} we have that $\Ht{\Rep{q}}$ is generated by the roots $\alpha \in \Delta^+_{T_Q}$ and 
$$\dim \Ht{\Rep{q}}[\beta]=\Big \{ (\gamma_\alpha) \in \mathbb N^{\Delta^+_{T_Q}} \ | \ \sum \gamma_\alpha \alpha =\beta \Big\}.$$ 
Hence $\rho$ is an epimorphism. Also we have that $\dim_{\mathbb C} L(T_Q) \leq |\Delta^+_{T_Q}|$ (\cite[Corollary 4.3]{kos}). Therefore by Proposition \ref{equalDim} $\rho$ is a monomorphism. Finally $\rho$ is an isomorphism.

\end{proof}

\section{Examples and final remarks}

\subsection{Some examples}

\begin{ex} 
\normalfont Let $Q$ be the following quiver
$$ 
\xymatrix 
{
	Q:1 \ar@{->}[r]^{a} & 2 \ar@{->}[r]^{b} & 3
}
$$
bound by $I=\langle b a \rangle$ the ideal of $\mathbb F_q Q$ generated by $ba$. Obviously $\mathbb F_q Q/I$ is a representation directed algebra. The corresponding quadratic form is:
\begin{equation}\label{exForm}
	T_Q(\beta)=\beta_1^2+\beta_2^2+\beta_3^2-\beta_1 \beta_2 - \beta_2 \beta_3 + \beta_1 \beta_3.
\end{equation}
It is not hard to see that the quantum Serre relations of $U_t^+(\mathfrak sl_4)$ are satisfied in $\Ht{\Rep{q}}$ where $t=+\sqrt{q}$. Moreover we have the following additional relations:
\begin{equation*}
\textrm{ad}^t_{e_1}(e_3)=0, \quad \textrm{ad}^t_{e_1}(\textrm{ad}^t_{e_2}(e_3))=0.
\end{equation*}
These relations are due to the fact that $\alpha_3$, $\alpha_2+\alpha_3$ are roots but $\alpha_1+\alpha_3$ and $\alpha_1+\alpha_2+\alpha_3$ are not the roots of the form \eqref{exForm}. Denoting $[x,y]_t=txy-yx$ to be the twisted commutator, we have that 
$$
 \mathbf{H}_{\Rep{q}}^{tw} \simeq \mathbb{C}\langle e_i\ |\ i=1,2,3\rangle / \mathcal R.
$$
where $\mathcal R$ is the two-sided ideal generated by the usual quantum Serre relations for $U_t^+(\mathfrak{sl}_4)$ but with the relation $[e_3,e_1]=0$ replaced by $[e_3,e_1]_t=0$ and one extra relation $$[e_1,[e_2,e_3]_t]=0.$$ This computation follows from the fact that the positive root system of $U_t^+(T_Q)$ is $\Delta^+(\mathfrak{sl}_4)\backslash \{\alpha_1+\alpha_2+\alpha_3\}.$

This description of $\Ht{\Rep{q}}$ is a little bit different than the one obtained in \cite[Example 5.9]{ChDe} for the same quiver $Q$ and ideal $I$ but with slightly different twisting. 
\end{ex}

\begin{ex}
\normalfont
More generally let $Q$ be the following quiver
$$ 
\xymatrix 
{
	Q:1 \ar@{->}[r]^{a_1} & 2 \ar@{->}[r]^{a_2} & \dots  \ar@{->}[r]^{a_n} & n
}
$$
bound by the ideal $I=\langle a_n\dots a_1 \rangle$ of the algebra $\mathbb F_q Q$. By similar observations we have that 
$$
 \Ht{\Rep{q}} \simeq \mathbb{C}\langle e_i\ |\ i=1,\dots,n\rangle / \mathcal R.
 $$
 where $\mathcal R$ is the two-sided ideal generated by the usual quantum Serre relations for $U_t^+(\mathfrak{sl}_{n+1})$ but with the relation $[e_n,e_1]=0$ replaced by $[e_n,e_1]_t=0$ and one extra relation $$[e_1,[e_2,\dots,[e_{n-1},e_n]_t \dots]_t].$$ Similar to the previous case, this follows since the positive root system of $U_t^+(T_Q)$ is $\Delta^+(\mathfrak{sl}_n)\backslash \{\alpha_1+\alpha_2+\cdots +\alpha_n\}.$	
\end{ex}

\begin{ex} \label{exRomb}
\normalfont

\normalfont Let $Q$ be the following quiver
$$ 
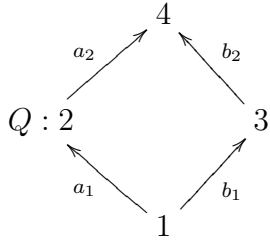
\xymatrix 
{
	& 4 &\\
	Q:2 \ar@{->}[ru]^{a_2} & & 3 \ar@{->}[lu]_{b_2}\\
	&\ar@{->}[lu]^{a_1}  1  \ar@{->}[ru]_{b_1}
}
$$
bound by the ideal $I=\langle a_2a_1- b_2b_1 \rangle$ in $\mathbb F_qQ$. We have $$T_Q(\beta)=\beta_1^2+\beta_2^2+\beta_3^2+\beta_4^2-\beta_1\beta_2-\beta_1\beta_3-\beta_2\beta_4-\beta_3\beta_4+\beta_1\beta_4.$$
In this case apart from the standard relations in $\Ht{\Rep{q}}$, which came from considering $Q$ as unbound quiver, we have the following extra relations:
$$
	[e_4,e_1]_t=0, \quad [e_1,[e_2,e_4]_t]=0, \quad [e_1,[e_3,e_4]_t]=0.
$$ 
This follows from a similar computation of the positive roots of the form $T_Q$.
\end{ex}
\begin{ex}
\normalfont Let $Q$ be the same quiver as in Example \ref{exRomb} bound by  ideal $I=\langle a_2a_1, b_2b_1 \rangle$. In this case we have $$T_Q(\beta)=\beta_1^2+\beta_2^2+\beta_3^2+\beta_4^2-\beta_1\beta_2-\beta_1\beta_3-\beta_2\beta_4-\beta_3\beta_4+2\beta_1\beta_4.$$
Hence the extra relations will have the following form 
\begin{align*}
  [e_1,[e_2,e_4]_t]_{t^{-1}}&=0, \quad [e_1,[e_3,e_4]_t]_{t^{-1}}=0 \\				
  [e_4,e_1]_{t^2}&=0,   \quad [e_1,[e_2,[e_3,e_4]_t]_t]=0.
\end{align*}
As before, this follows from computing the positive roots of the form $T_Q$.

\end{ex}

\subsection{Commutative representations of quivers over $\mathbb F_1$ and Hall algebra}

As the limiting case of our construction one can also study representations of quivers with commutativity conditions over the so-called field with one element: $\mathbb F_1$. Such a field is not defined per se, but there is agreement on what should be the definition and basic properties of the category of vector spaces over $\mathbb{F}_1$ as a limiting case of the categories of vector spaces over $\mathbb{F}_q$ (see for example \cite{matt}).

Suppose that $Q$ is without oriented cycles and loops. The category $\Rep{1}$ of finite-dimensional representations of $Q$ over $\mathbb F_1$ is defined similarly to $\textrm{Rep}_Q(\mathbb F_q)$ (see \cite[Section 4]{matt}). It is straightforward to define the representation of quiver over $\mathbb F_1$ which satisfy given relations. In the case when $Q$ has only commutativity relations it is straightforward to show that any indecomposable object in $\textrm{Rep}_Q(\mathbb F_q)$ is one-dimensional (similar statement to \cite[Theorem 5.1]{matt}). 

M.Szczesny in \cite{matt} defined the Hall algebra, $\mathbf H_{\Rep{1}}$, of $\Rep{1}$ for the case when $Q$ is unbound following Ringel's definition. A slightly modified construction can be applied in the bound case as well, and we use the same notation as the unbound case. Moreover, let $\mathbb K=\{0,1\}$ be the one-dimensional space over $\mathbb F_1$. Denote by $S_i$ a simple representation of $Q$ at vertex $i\in Q_0$, i.e. $(S_i)_i=\mathbb K$, $(S_i)_j={0}$ for $j\neq i$ and $(S_i)_\rho=0$ for all $\rho\in Q_1$. Then one shows the following

\begin{rem}
Let $Q$ be a quiver without oriented cycles, loops or multiple arrows bound by all possible commutativity relations. The assigment $e_i \mapsto [S_i]$ gives rise to an epimorphism of associative algebras $U^+_1(T_Q)$ and \normalfont $\mathbf H_{\Rep{1}}$.
\end{rem}

\end{document}